\documentclass[10pt, a4paper]{article}
\usepackage{graphicx,latexsym, amsmath,amsfonts}
\usepackage{amsthm}
\usepackage{enumerate}
\usepackage{bbm}
\usepackage{subfig}

\usepackage{epsfig}
\usepackage{color}
\usepackage[nocompress]{cite}

\begin{document}

	
	\renewcommand{\d}{d}
	
	\newcommand{\E}{\mathbb{E}}
	\newcommand{\PP}{\mathbb{P}}
	\newcommand{\DD}{\mathbb{D}}
	\newcommand{\R}{\mathbb{R}}
	\newcommand{\cD}{\mathcal{D}}
	\newcommand{\cF}{\mathcal{F}}
	\newcommand{\cK}{\mathcal{K}}
	\newcommand{\N}{\mathbb{N}}
	\newcommand{\fracs}[2]{{ \textstyle \frac{#1}{#2} }}
	\newcommand{\sign}{\text{sign}}
		\newcommand{\ga}{\mathfrak{g}}
		\newcommand{\sk}{\mathfrak{s}}
	
	\newtheorem{theorem}{Theorem}[section]
	\newtheorem{lemma}[theorem]{Lemma}
	\newtheorem{assumption}[theorem]{Assumption}
	\newtheorem{coro}[theorem]{Corollary}
	\newtheorem{defn}[theorem]{Definition}
	\newtheorem{assp}[theorem]{Assumption}
	\newtheorem{expl}[theorem]{Example}
	\newtheorem{prop}[theorem]{Proposition}
	\newtheorem{proposition}[theorem]{Proposition}
	\newtheorem{corollary}[theorem]{Corollary}
	\newtheorem{rmk}[theorem]{Remark}
	\newtheorem{notation}[theorem]{Notation}

	\def\a{\alpha} \def\g{\gamma}
	\def\e{\varepsilon} \def\z{\zeta} \def\y{\eta} \def\o{\theta}
	\def\vo{\vartheta} \def\k{\kappa} \def\l{\lambda} \def\m{\mu} \def\n{\nu}
	\def\x{\xi}  \def\r{\rho} \def\s{\sigma}
	\def\p{\phi} \def\f{\varphi}   \def\w{\omega}
	\def\q{\surd} \def\i{\bot} \def\h{\forall} \def\j{\emptyset}
	
	\def\be{\beta} \def\de{\delta} \def\up{\upsilon} \def\eq{\equiv}
	\def\ve{\vee} \def\we{\wedge}
	
	\def\t{\tau}
	
	\def\F{{\cal F}}
	\def\T{\tau} \def\G{\Gamma}  \def\D{\Delta} \def\O{\Theta} \def\L{\Lambda}
	\def\X{\Xi} \def\S{\Sigma} \def\W{\Omega}
	\def\M{\partial} \def\N{\nabla} \def\Ex{\exists} \def\K{\times}
	\def\V{\bigvee} \def\U{\bigwedge}
	
	\def\1{\oslash} \def\2{\oplus} \def\3{\otimes} \def\4{\ominus}
	\def\5{\circ} \def\6{\odot} \def\7{\backslash} \def\8{\infty}
	\def\9{\bigcap} \def\0{\bigcup} \def\+{\pm} \def\-{\mp}
	\def\<{\langle} \def\>{\rangle}
	
	\def\lev{\left\vert} \def\rev{\right\vert}
	\def\1{\mathbf{1}}

	\newcommand\wD{\widehat{\D}}
	\newcommand\EE{\mathbb{E}}
	
	\newcommand{\ls}[1]{\textcolor{red}{\tt Lukas: #1 }}

	\title{ \bf On the convergence order  of the Euler scheme for scalar SDEs with H\"older-type diffusion coefficients}

	\author{Annalena Mickel \footnote{former member of Mathematical Institute and   DFG RTG 1953, 
			University of Mannheim,  B6, 26, D-68131 Mannheim, Germany}
		\and   Andreas Neuenkirch  \footnote{Mathematical Institute, 
			University of Mannheim,  B6, 26, D-68131 Mannheim, Germany, \texttt{aneuenki@mail.uni-mannheim.de}} 
	}

	\date{\today}

	\maketitle
	
	\begin{abstract}
	We study the Euler scheme for scalar non-autonomous stochastic differential equations, 
     whose diffusion coefficient is not globally Lipschitz but a fractional power of a globally Lipschitz function. We analyse the strong error and establish a criterion, which relates the convergence order of the Euler scheme to an inverse moment condition for the diffusion coefficient.
	Our result in particular applies to Cox-Ingersoll-Ross-, Chan-Karolyi-Longstaff-Sanders- or Wright-Fisher-type stochastic differential equations and thus provides a unifying framework.

		\medskip
		\noindent \textsf{{\bf Key words: } \em scalar SDEs, Euler scheme, non-standard assumptions, fractional power, $L^1$-error, inverse moment condition \\
		}
	\end{abstract}
	
	\section{Introduction and Main Result}
	Consider the scalar non-autonomous stochastic differential equation (SDE)
	\begin{equation}\label{SDE}
		dX_t = a(t,X_t)dt + c(t,X_t)dW_t,  \quad t\in[0,T], \qquad X_0=x_0 \in \mathbb{R},
	\end{equation} where $W=(W_t)_{t \in [0,T]}$ is a Brownian motion  and the drift and diffusion coefficients $a, c:  [0,T] \times \mathbb{R} \rightarrow \mathbb{R}$ satisfy suitable assumptions. Explicit solutions of such SDEs are rarely known and a simple and popular numerical method for this equation is the Euler scheme.
For an equidistant discretization with $N\in\mathbb{N}$ time steps, i.e.,
\begin{displaymath}
	t_k = k \frac{T}{N} , \quad k=0, \ldots,N,
\end{displaymath}
the equidistant Euler scheme for \eqref{SDE} is given by $x^{(N)}_k=x_k$ with
\begin{equation} \label{euler}
		x_{k+1} = x_k + a(t_k,x_k)(t_{k+1}-t_k)+c(t_k,x_k)(W_{t_{k+1}}-W_{t_{k}}), \qquad k=0, \ldots, N-1.
\end{equation}

	It is well known that the Euler scheme  \eqref{euler} for SDE \eqref{SDE} has strong convergence order $1/2$, i.e., one has
	$$ \limsup_{N \rightarrow \infty} \,\, N^{1/2} \, \sup_{k=0, \ldots, N} \mathbb{E}  \left| X_{t_k}- x_k^{(N)} \right|  < \infty
	, $$
	 if the drift and diffusion coefficient satisfy the following standard assumptions  (HLG) and (L).

	 \begin{assumption} \label{ass-1}
	 We will use the following  assumptions on  functions $f: [0,T] \times \mathbb{R} \rightarrow \mathbb{R}$, which are measurable.
	\begin{itemize} 
		\item[(HLG)] There exists a constant $K>0$ such that
		$$|f(t,x)-f(s,x)| \leq K (1+|x|) |t-s|^{1/2}$$
		for all $s,t \in [0,T]$ and all $x \in \mathbb{R}$.
		\item[(L)] There exists a constant $K>0$ such that
		$$|f(t,x)-f(t,y)| \leq K |x-y|$$
		for all $t \in [0,T]$ and all $x,y \in \mathbb{R}$.
		\item[(P)] We have 
		$$ f(t,x) \geq 0$$
			for all $t \in [0,T]$ and all $x \in \mathbb{R}$.
	\end{itemize}
	\end{assumption}

	In recent years, the analysis of numerical methods for SDEs under non-standard assumptions has become a very active research field, see Subsection \ref{subsec:overview} for an overview.
	In this article, we consider SDEs where the drift coefficient satisfies the standard assumptions, but the diffusion coefficient is a fractional power of a standard diffusion coefficient. More precisely, we will study the SDE
	\begin{equation}\label{SDE-g}
		dX_t = a(t,X_t)dt + \sigma(t,X_t)^{\ga}dW_t,  \quad t\in[0,T], \qquad X_0=x_0 \in \mathbb{R},
	\end{equation} where
$a,\sigma:[0,T] \times \mathbb{R} \rightarrow \mathbb{R}$ satisfy  (HLG) and (L), $\sigma$ satisfies additionally (P)
 and we have $\ga \in [1/2,1)$.
 Since $a$ and $\sigma^{\ga}$  are continuous and of linear growth, the existence of a weak solution to Equation \eqref{SDE-g} follows from the works of Skorokhod  (see \cite{Sko61,Sko62} and  also  \cite{krylov-book}). Since moreover 
 $$ |\sigma(t,x)^{\ga}-\sigma(t,y)^{\ga}| \leq K^{\ga} |x-y|^{\ga}$$
 for all $t \in [0,T]$ and all $x,y\in \mathbb{R}$
 and
 $$ \int_0^{\varepsilon} x^{-2\ga}dx= \infty $$
 for all $\varepsilon>0$, we have strong uniqueness and consequently existence of a unique strong solution to Equation  \eqref{SDE-g} by the results of Yamada and Watanabe (see, e.g., Proposition  2.13 and Corollary 3.23 in Chapter V of \cite{KS}).

 Typical examples for SDE \eqref{SDE-g} are the non-autonomous Cox-Ingersoll-Ross (CIR) process (see,  e.g., \cite{CIR})
 $$
 	dX_t=\kappa(t)\left(\lambda(t)-X_t\right)dt+\theta(t)\sqrt{X_t^+}dW_t, 
$$ the non-autonomous Chan-Karolyi-Longstaff-Sanders (CKLS) process (see, e.g., \cite{CKLS})
 $$
dX_t=\kappa(t)\left(\lambda(t)-X_t\right)dt+\theta(t) \left( X_t^+ \right)^{\ga}dW_t, 
$$ with $\ga \in (1/2,1)$ 
 and the non-autonomous   Wright-Fisher (WF)  process (see, e.g., \cite{Ethier_Kurtz})
  $$
 dX_t=\kappa(t)\left(\lambda(t)-X_t\right)dt+\theta(t)\sqrt{(X_t(1-X_t))^+}dW_t, 
 $$with $\kappa, \lambda, \theta: [0,T] \rightarrow [0,\infty)$. These SDEs have applications in various fields, e.g., in biology and finance, to mention a few.

\medskip

To formulate our main result, we will need the following assumption:
	 \begin{assumption}\label{ass-k}  Assume that $a, \sigma$ satisfy  (HLG) and (L) and that $\sigma$ additionally satisfies (P). 
	 	Moreover let $\ga\in [1/2,1)$ and assume that there exists  $\sk \in [0,1-\ga]$ such that
	\begin{equation} \label{ass-k-thm}  \int_0^T  \mathbb{E} \left[ \sigma(s,X_s)^{2(\ga+\sk-1)} \right]ds < \infty. \end{equation}
	\end{assumption}

	Our main result reads then as follows: 
	\begin{theorem} \label{thm:up-bound}  Let $\ga\in [1/2,1)$ be given and assume that Assumption \ref{ass-k} holds for a fixed $\sk \in [0,1-\ga]$. Moreover, let $X=(X_t)_{t \in [0,T]}$ be the solution  of \eqref{SDE-g} and let $(x_k)_{k=0, \ldots, N}$ be  given by the Euler scheme \eqref{euler} with diffusion coefficient $c=\sigma^{\ga}$.   Then we have
		\begin{equation*}
		\limsup_{N \rightarrow \infty} \,\, N^{\lambda} \, \sup_{k=0, \ldots, N} \mathbb{E} \left| X_{t_k}- x_k^{(N)} \right| =0	
	\end{equation*}		
		for all $\lambda < 1/2-\sk$.
	\end{theorem}

How to interpret condition  \eqref{ass-k-thm}? For this note first that the irregularity of the diffusion coefficient $\sigma^{\ga}$ arises from the set
$$ \mathcal{I}_{\sigma}= \{ (t,x) \in [0,T] \times \mathbb{R}: \, \sigma(t,x)=0\}.$$
For $(t_*,x_*) \in \mathcal{I}_{\sigma}$ and a neighbourhood $U_{*}$ of $(t_*,x_*)$  the map $ U_{*} \ni (t,x) \mapsto \sigma(t,x)^{\ga} \in [0,\infty)$ is only H\"older continuous of order $\ga$ and not Lipschitz continuous.
Moreover, since   $\ga+\sk \leq 1$ the exponent in  condition \eqref{ass-k-thm} is zero or negative. Thus, this condition is a weighted measure, how much the solution $X$ takes values in or close to the set $\mathcal{I}_{\sigma}$. In the case $\sk=0$, that is
\begin{equation*}   \int_0^T  \mathbb{E} \left[ \sigma(s,X_s)^{2\ga-2} \right] ds < \infty, \end{equation*} the process $X$ is sufficiently apart from the set $\mathcal{I}_{\sigma}$ and  we have directly convergence order $1/2-\varepsilon$ for all $\varepsilon >0$. If on the other hand
\begin{equation*}   \int_0^T  \mathbb{E} \left[ \sigma(s,X_s)^{2\ga-2} \right] ds = \infty, \end{equation*}
 the process $X$ is too close to the set $\mathcal{I}_{\sigma}$ and  we need a compensation factor $\sk$ and obtain convergence order $1/2-\sk -\varepsilon$.

Finally, let us comment on three limiting cases: 
\begin{itemize}
	\item[(i)] For $\ga =1$ condition  \eqref{ass-k-thm} would be  trivially satisfied for $\sk=0$ and we would (formally) obtain convergence  order  $1/2-\varepsilon$ in accordance with the classic result for the Euler scheme.
\item[(ii)] If $\sigma$ is uniformly elliptic,  i.e.,
$$ \inf_{(t,x) \in [0,T] \times \mathbb{R}} \sigma(t,x) > 0,$$
then condition \eqref{ass-k-thm} is trivially satisfied for $\sk=0$  and we obtain convergence  order  $1/2-\varepsilon$. This is again in accordance with the classic result for the Euler scheme, since $\sigma^{\ga}$ satisfies (HLG) and (L), if $\sigma$ is uniformly elliptic. \item[(iii)] If $\ga >1/2$ we can choose $\sk=1-\ga$. Then  condition \eqref{ass-k-thm} is satisfied, since $\ga+\sk-1=0$, and we obtain convergence order $\ga- 1/2-\varepsilon$. This recovers  in our setup the results of  \cite{Yan} and  \cite{GR} for autonomous SDEs.
\end{itemize}

How to verify condition \eqref{ass-k-thm}? For this we present two approaches in Subsection \ref{sec:examp-ito} and Subsection \ref{sec:examp-comp-time}.

\medskip

The remainder of this article is structured as follows:
In the next two subsections, we give an overview of previous results in the literature and outline our notation. Section \ref{sec:examples} applies our main result to several SDEs, in particular to the above mentioned CIR, CKLS and WF processes. Section \ref{proof-main-result} is devoted to the proof of our main result, while in the Appendix we collect some general auxiliary results.

	\subsection{Overview of previous results}\label{subsec:overview}
	
	 Since we have $\gamma\in\left[\frac{1}{2},1\right)$, the diffusion coefficient of \eqref{SDE-g} is not globally Lipschitz continuous and therefore standard convergence results for the Euler scheme do not apply.
	 	Following the pioneering articles \cite{GyKr,GY,HMS,Yan,HaKl}, the strong approximation of the Euler scheme for SDEs with non-standard coefficients has been extensively studied\footnote{To improve the readability, we will slightly abuse notation and write in this overview {\it convergence order $\lambda$} also for results which established {\it convergence order $\lambda -\varepsilon$ for all $\varepsilon>0$} for a given $\lambda>0$.}.
	 	
	 	\medskip
	 	
	 	In the case of an irregular drift coefficient, there has been a series of articles including \cite{ngo2016,ngo2017a,ngo2017b,LS3,Bao,NeSzSz,GLN,DG,NeuSz,DA,le2022taming} that has culminated so far in the works \cite{MGY1} and \cite{DGL}, which established under mild assumptions on the drift coefficient the classical convergence order $1/2$ for the Euler scheme. More precisely, for scalar autonomous SDEs with piecewise Lipschitz drift and a globally Lipschitz diffusion coefficient, which does not vanish in the discontinuity points of the drift, an $L^p$-error of order $1/2$ is obtained in \cite{MGY1}, while in the multi-dimensional case for a bounded drift and an elliptic, bounded $\mathcal{C}^2_b$-diffusion coefficient the same classical convergence order is obtained in  \cite{DGL}.
	 	
	 	\smallskip
	 	
	  Furthermore, the strong approximation by other schemes than the Euler scheme has been analyzed, e.g., in  \cite{LS1,LS2,Yaros3,MGY2,MGY3,gerencs_milstein}, and corresponding lower error bounds for the strong approximation of SDEs with (possibly) discontinuous drift coefficient have been studied in \cite{HeHeMG}, \cite{MGY3} and \cite{ellinger_sharp}.
	 
\medskip

	 The case of a general SDE with a non-Lipschitz diffusion coefficient has received less attention. While the works \cite{ngo2016,ngo2017a,ngo2017b} also allow weaker assumptions on the diffusion coefficient, the strongest results so far are due to \cite{Yan} and \cite{GR}.

	 \smallskip
	 
	 In \cite{Yan} the following H\"older- and Lipschitz-continuity assumptions on the coefficients are used: There exist $K>0, \beta_1,\beta_2 \in (0,1], \ga \in (1/2,1]$ such that
	 \begin{align*}
	 |a(t,x)-a(s,y)| \leq K  \left( |t-s|^{\beta_1} + |x-y| \right), \qquad 
	 	|c(t,x)-c(s,y)| \leq K  \left( |t-s|^{\beta_2} + |x-y|^{\ga} \right) 
	 \end{align*}
	for all $s,t \in [0,T],x,y \in \mathbb{R}$. Then, Theorem 4.1 in \cite{Yan} gives for the Euler scheme \eqref{euler} the $L^1$-convergence order  $\min \{\beta_1, \ga - 1/2, (2-1/\ga)\beta_2\}$ in the discretization points.

	 \smallskip

	 The article \cite{GR} analyses the Euler-type polygonal approximation 
	 	\begin{equation}\label{SDE_polygon}
	 	dY_t = a(t,Y_{\eta(t)})dt + c(t,Y_{\eta(t)})dW_t,  \quad t\in[0,T], \qquad Y_0=x_0 \in \mathbb{R},
	 \end{equation}
 where $\eta(t)=\max\{t_k\in\{t_0,t_1,...,t_N\}: t_k\le t\}$.
	 For autonomous SDEs this approximation coincides in the discretization points with the Euler scheme \eqref{euler}, while for general non-autonomous coefficients this scheme is typically not implementable.
	   The assumptions  in \cite{GR} are as follows:
	 We have
	 $ a=b+f$, where  for all $t \in [0,T]$ the function $f(t,\cdot)$ is monotonically decreasing, 
	 $$ \sup_{t \in [0,T]}  \left(|b(t,0)| + |c(t,0)| \right) < \infty$$ and there exists  $K>0, \beta_3 \in (0,1], \ga \in (1/2,1]$ such that
	  \begin{align*}
	 	|b(t,x)-b(s,y)| \leq K  |x-y| , \quad 
	 	|c(t,x)-c(s,y)| \leq K   |x-y|^{\ga},  \quad  |f(t,x)-f(s,y)|\leq K   |x-y|^{\beta_3}
	 \end{align*}
	 for all $t \in [0,T], x,y \in \mathbb{R}$. Under these assumptions,  Theorem 2.1 in \cite{GR} gives that the $L^1$-error of \eqref{SDE_polygon} is bounded by $C / \log(N)$ for a constant $C>0$ if $\ga=1/2$ and is of order $\min\{ \ga -1/2, \beta_3/2\}$ for $\ga \in (1/2,1]$.

	 	\smallskip

	On the other hand, the classical autonomous CIR process has a H\"older-1/2-diffusion coefficient and has received  a lot of attention in recent years, see, e.g., \cite{Mario,JentzenHefter,MiNe} and  the references therein.
	In particular, it turned out that Euler schemes  attain  for this SDE  $L^1$-convergence order $1/2$, if the Feller index $\nu:= \frac{2 \kappa \lambda}{\theta^2}$ satisfies $\nu>1$, see \cite{MiNe}. In this case our Assumption \ref{ass-k} is satisfied with $\sk=0$.
	Remarkably, this is a much better decay of the error than the logarithmic one provided in \cite{GR}. Similarly, for CKLS-type SDEs with $\ga \in (1/2,1)$ and a symmetrized Euler scheme the work \cite{BER} also establishes convergence order $1/2$.
	
	\medskip

	Thus, Theorem \ref{thm:up-bound} provides a unifying framework and also  a criterion, i.e., condition \eqref{ass-k}, which explains and unites these  different error bounds for the  Euler scheme.

	\subsection{Notation}
	As already mentioned, we will work with an equidistant discretization
	\begin{displaymath}
		t_k = k \Delta t, \quad k=0, \ldots, N,
	\end{displaymath}
	with $\Delta t= T/ N$ and $N \in \mathbb{N}$. Furthermore, we define $n(t):=\max\{k\in \{0,...,N\}:t_k\le t \}$ and $\eta(t):=t_{n(t)}$.  Constants whose values depend only on $x_0,T$, $a,\sigma $, $\ga$ will be denoted in the following by $C$, regardless of their value. Other dependencies will be denoted by subscripts, i.e., $C_{p}$ means that this constant depends additionally on the parameter $p$. Moreover, the value of all these constants can change from line to line. We will work on  a  filtered probability space $(\Omega, \cF, (\cF_t)_{t\in[0,T]}, {P})$ where the filtration satisfies the usual conditions, and (in-)equalities between random variables or random processes are understood ${P}$-a.s.~unless mentioned otherwise.

	Finally, for $p \in \mathbb{N}$ we denote by $\mathcal{C}^{p}(M;\mathbb{R})$ the  set of $p$-times continuously differentiable functions from the set $M$ to $\mathbb{R}$ and by  $\mathcal{C}^{p}_b(M;\mathbb{R})$ the set of continuously differentiable functions from  $M$ to $\mathbb{R}$, whose derivatives are bounded. Moreover, for $\gamma \in (0,1)$, we denote by $\mathcal{C}^{\gamma}(M;\mathbb{R})$ the  set of $\gamma$-H\"older continuous functions from $M$ to $\mathbb{R}$.

	\section{Examples}\label{sec:examples}
 Condition \eqref{ass-k-thm} from Assumption \ref{ass-k} gives structural insight in the convergence behaviour of the Euler scheme, but how to verify it? Here we present two approaches. The first one is based on the Feller test and It\=o's lemma.
The second one is based on a time change and the comparison lemma for SDEs.
 
 \subsection{Autonomous SDEs on domains}\label{sec:examp-ito}
 
Here we consider the autonomous equation
 \begin{equation}\label{auto_SDE}
 	dX_t = a(X_t)dt + \sigma(X_t)^{\ga}dW_t, \quad t \in [0,T], \qquad X_0 =x_0,
 \end{equation}
 and we assume that  $a\in C^{1}_b(\mathbb{R};\mathbb{R})$ and $\sigma \in C^{2}_b(\mathbb{R};[0,\infty))$. To apply It\=o's lemma to derive an explicit criterion for condition  \eqref{ass-k-thm} we need that $\sigma^{\ga}$ is differentiable on the support of $X$. To this end, we assume that there exists
 $I=(l,r)$ with $-\infty\le l\le r\le \infty$
 such that
 \begin{align*}
 	\forall x\in I:\quad	\sigma(x)>0,
 \end{align*}
 as well as
 $$ P(X_t \in I, \, t\in [0,T])=1.$$
 The latter condition can be checked, e.g., by the Feller test, see Theorem \ref{thm:feller-test} in the Appendix.
 
Under these assumptions it is well known that the Euler scheme has convergence order $1/2$ in law, that is
 \begin{equation*}
 	\lim_{N \rightarrow \infty} \,\, N^{1/2} \, \sup_{k=0, \ldots, N} \big| X_{t_k}- x_k^{(N)} \big|  \stackrel{\mathcal{L}}{=} \max_{t \in [0,T]} |U_t|,
 \end{equation*}		
 where the limiting process $U$ is the solution of the SDE
 \begin{align}\label{sde:limiting}
 	dU_t & =a'(X_t)U_t dt + \ga  \sigma^{\ga-1}(X_t)\sigma'(X_t)U_t dW_t + 	\frac{\ga }{\sqrt{2}}   \sigma(X_t)^{2\ga-1}\sigma'(X_t) d B_t,  \qquad U_0=0,
 \end{align}
 where $B=(B_t)_{t \in [0,T]}$ is another Brownian motion which is independent of $W$.
 See \cite{Nzahle} and \cite{Protter_etal}. Thus it is reasonable to expect that in this case the Euler scheme has also strong convergence order $1/2$. The limiting SDE \eqref{sde:limiting} also illustrates that the bottleneck is the error propagation and not the local error, since we have the exponent $\ga-1$ in the error propagation term  and the less restrictive exponent $2\ga-1$ in the local error term. 
 
 \begin{proposition} 
 	Let $\ga \in [1,2,1)$, $a\in C^{1}_b(\mathbb{R};\mathbb{R})$ and $\sigma \in C^{2}_b(\mathbb{R};[0,\infty))$.
 	Moreover let  $I=(l,r), -\infty\le l\le r\le \infty$, and assume that the unique solution of the SDE
 	\begin{equation*}
 		dX_t = a(X_t)dt + \sigma(X_t)^{\ga}dW_t,  \quad t\in[0,T], \qquad X_0=x_0 \in I,
 	\end{equation*}
 	satisfies 
 	$$ P(X_t \in I, \, t\in [0,T])=1$$
 	as well as
 	\begin{align*}
 		\forall x\in I:\quad	\sigma(x)>0.
 	\end{align*}
 	Finally, assume that
 	\begin{align}\label{ass:ito}
 		\inf_{x \in I} \left[   \frac{\sigma'(x)a(x)}{\sigma(x)} +  \frac{1}{2} \frac{\sigma''(x) \sigma(x)^{2\ga}}{\sigma(x)}  +\left(\ga- \frac{3}{2} \right)  \frac{(\sigma'(x))^2 \sigma(x)^{2\ga}}{\sigma(x)^2}   \right] > - \infty.
 	\end{align}

 	Then we have
 	\begin{equation*}
 		\limsup_{N \rightarrow \infty} \,\, N^{\lambda} \, \sup_{k=0, \ldots, N} \mathbb{E} \left| X_{t_k}- x_k^{(N)} \right| =0	
 	\end{equation*}		
 	for all $\lambda < 1/2$.
 	
 \end{proposition}
 \begin{proof}
 	Let $k \in \mathbb{N}$ such that $k > \sigma(x_0)^{2(\ga-1)}$ and set 
 	$$ \tau_k= \inf \{t \in [0,T]: \, \sigma(X_t)^{2(\ga-1)} \geq k \}$$ with the convention $\inf \emptyset =T$. An application of It\=o's lemma and the local martingale property of the It\=o integral give
 \begin{align*}
 	\mathbb{E} \left( \sigma(X_{t\wedge \tau_k})^{2(\ga-1)} \right) &= \sigma(x_0)^{2(\ga-1)}
 	+ \mathbb{E} \left( \int_0^{t\wedge \tau_k} h_{a,\sigma,\ga}(X_s) ds \right)
 \end{align*}	
 	with
 	\begin{align*}
h_{a,\sigma,\ga}(x)& =  (2\ga-2) \sigma(x)^{2\ga-2}& \\ & \qquad \times \left( \sigma'(x)\sigma^{-1}(x)a(x)+\left(\ga- \frac{3}{2} \right)\sigma(x)^{2\ga-2} (\sigma'(x))^2+ \frac{1}{2}\sigma(x)^{2\ga-1} \sigma''(x) \right).
 	\end{align*}
 Assumption \eqref{ass:ito} now yields the existence of constant $C>0$, which is in particular independent of $k$, such that
 $$ h_{a,\sigma,\ga}(x) \leq C \sigma(x)^{2(\ga-1)} $$ for all $x \in I$. 
 Gronwall's lemma gives then the existence   of another constant $C>0$, which is in particular independent of $k$, such that
$$ \sup_{t \in [0,T]}	\mathbb{E} \left( \sigma(X_{t\wedge \tau_k})^{2(\ga-1)} \right) \leq C. $$
 Assumption \eqref{ass-k-thm} with $\sk=0$ now follows by taking $k\rightarrow \infty$.
 \end{proof}

 \subsection{Prototype SDEs}\label{sec:examp-comp-time}
 In this subsection we consider the SDEs
 \begin{equation}	\label{SDE_proto}
 	dX_t = a(t,X_t)dt + c_i(t,X_t)dW_t,  \qquad t\in[0,T], 
\end{equation}
for $i=1,2,3$, where
$$
 	c_i(t,x)=\theta(t)h_{i}(x)$$
with $\theta \in C^{1/2}([0,T];(0,\infty))$ and
$$ h_1(x)=\sqrt{x^+}, \qquad h_2(x) =\sqrt{(x(1-x))^+}, \qquad h_3(x)=(x^+)^{\ga},$$
where $\ga \in (1/2,1)$.

Theses SDEs are extensions of the classical CIR, CKLS and WF processes. For verifying Assumption \ref{ass-k}, the following time change procedure will be helpful. So, let 
$$\Theta:[0,T] \rightarrow \mathbb{R}, \qquad \Theta(t)=\int_0^t \theta^2(s) ds. $$ Then $\Theta$ is strictly increasing and has an inverse $A:=\Theta^{-1}$. Now consider the time-changed process
$$ \widetilde{X}_t=X_{A(t)}, \qquad t \in [0,A^{-1}(T)].$$
Since $X$ satisfies 
$$X_t=x_0 + \int_0^t a(s,X_s)ds + \int_0^t \theta(s) h_i(X_s) dW_s, \qquad t \in [0,T],$$
we have
\begin{align*} \widetilde{X}_t=x_0 + \int_0^{A(t)} a(s,X_s)ds + \int_0^{A(t)} \theta(s) h_i(X_s) dW_s, \qquad t \in [0,T]. \end{align*}
Now set 
$$ \widetilde{W}_t = \int_0^{A(t)} \theta(s) dW_s, \qquad t \in [0,A^{-1}(T)].$$
Then $\widetilde{W}$ is a Brownian motion on $[0,A^{-1}(T)]$ and the rules for the time change for Riemann and It\=o integrals, see Proposition \ref{prop:time-change} in the Appendix, yield that
\begin{align*} \widetilde{X}_t & = x_0 +\int_0^{t} \frac{a(A(s),\widetilde{X}_{s})}{\theta(A(s))^2}ds + \int_0^{t} h_i(\widetilde{X}_{s}) d\widetilde{W}_s, \qquad t \in [0,A^{-1}(T)].  \end{align*}
This enables us to exploit the classical bounds for the inverse moments of the CIR, CKLS and WF processes. See Subsection \ref{subsec:classical_bounds} in the Appendix.

\smallskip

For $c_1$ and $c_2$ we obtain the following result:

\begin{proposition}\label{prop:cir-wf}
 (i) Let $x_0 >0$ and $\theta \in \mathcal{C}^{1/2}([0,T];(0,\infty))$. Moreover, assume that the drift coefficient
$a$ satisfies	(HLG), (L)
	 and assume furthermore that $$ \mu_0 := \min_{t \in [0,T]} \frac{a(t,0)}{\theta(t)^2} > 0.$$
	 Then the unique strong solution of 
	 $$   	dX_t = a(t,X_t)dt + \theta(t) \sqrt{X_t^+}dW_t,  \quad t\in[0,T], \qquad X_0=x_0,$$
	satisfies $ P( X_t \geq 0, t \in [0,T])=1$ and we have
			\begin{equation*}
		\limsup_{N \rightarrow \infty} \,\, N^{\lambda} \, \sup_{k=0, \ldots, N} \mathbb{E} \left| X_{t_k}- x_k^{(N)} \right| =0	
	\end{equation*}		
	for all $$\lambda < \min \{1/2, \mu_0\}.$$
	(ii)  Let $x_0 \in (0,1)$ and $\theta \in \mathcal{C}^{1/2}([0,T];(0,\infty))$. Moreover, assume that the drift coefficient
	$a$ satisfies	(HLG), (L)
	and assume furthermore that $$ \mu_0 := \min_{t \in [0,T]} \frac{a(t,0)}{\theta(t)^2} > 0 \qquad \textrm{and} \qquad 
	\mu_1 := - \max_{t \in [0,T]} \frac{a(t,1)}{\theta(t)^2}  > 0.$$
	Then the unique strong solution of 
	$$   	dX_t = a(t,X_t)dt + \theta(t) \sqrt{ (X_t(1-X_t))^+} dW_t,  \quad t\in[0,T], \qquad X_0=x_0,$$
	satisfies $ P( X_t \in [0,1], t \in [0,T])=1$ and we have
	\begin{equation*}
		\limsup_{N \rightarrow \infty} \,\, N^{\lambda} \, \sup_{k=0, \ldots, N} \mathbb{E} \left| X_{t_k}- x_k^{(N)} \right| =0	
	\end{equation*}		
	for all $$\lambda < \min \{1/2, \mu_0,\mu_1\}.$$
\end{proposition}

\begin{proof} (1) Let us first check the assumptions (HLG) and (L) on $a$ and $\sigma$ for both equations. For the drift coefficient there is nothing to verify.
	
	 For the diffusion coefficient in (i) note that $\sigma$ is given by
	$\sigma(t,x)=\theta(t)^{2}x^+$. So clearly, we have
	$$ |\sigma(t,x)-\sigma(t,y)| \leq  \left( \max_{\tau \in [0,T]} |\theta(\tau)|^{2} \right) |x^+-y^+| \leq \| \theta\|_{\infty}^2  |x-y|  $$
with
$$ \|\theta\|_{\infty} := \max_{ t \in [0,T] } |\theta(t)| < \infty$$
	and
		\begin{align*}
		|\sigma(t,x)-\sigma(s,x)| & \leq  2 \left( \max_{\tau \in [0,T]}  \theta(\tau) \right) |\theta(t)-\theta(s)| |x^+|  \\ & \leq 2  \| \theta \|_{\infty} \|\theta\|_{1/2} |t-s|^{1/2}|x| 
		\end{align*}
with
$$ \|\theta\|_{1/2} := \sup_{0 \leq s < t \leq T } \frac{|\theta(t)-\theta(s)|}{|t-s|^{1/2}}<\infty.$$ 

For (ii) note that $\sigma(t,x)= \theta(t)^2(x(1-x))^+$. In this case we have
	$$ |\sigma(t,x)-\sigma(t,y)| \leq  \max_{\tau \in [0,T]} |\theta(\tau)|^{2} |(x(1-x))^+-(y(1-y))^+| \leq \| \theta\|_{\infty}^2  |x-y|  $$
and
\begin{align*}
	|\sigma(t,x)-\sigma(s,x)| & \leq  2  \| \theta \|_{\infty} \|\theta\|_{1/2} |t-s|^{1/2}.
\end{align*}

 (2) Thus it remains to relate the inverse moment condition to  the quantities $\mu_0$ and $\mu_1$. We start with assertion (i). First note that
 $$ 0< \min_{t \in [0,T]} \theta(t) \leq  \max_{t \in [0,T]} \theta(t)  < \infty.$$
  Using this and the  time change described   above
 we have that
 $$ \sup_{t \in [0,T]} \mathbb{E} \left[  \left(\theta(t)\sqrt{ X_t^+} \right )^{p} \right] < \infty \quad \Longleftrightarrow  \sup_{t \in [0,A^{-1}(T)]}  \mathbb{E} \left[\widetilde{X}_t^{p/2} \right] < \infty $$
 for $p \in \mathbb{R}$, where 
	  \begin{align*} \widetilde{X}_t&= x_0 +\int_0^{t} \frac{a(A(s),\widetilde{X}_{s})}{\theta(A(s))^2}ds + \int_0^{t} \sqrt{\widetilde{X}_{s}^+} d\widetilde{W}_s, \qquad t \in [0,A^{-1}(T)].  \end{align*}
	  Recall that
	   $$ \mu_0 := \min_{t \in [0,T]} \frac{a(t,0)}{\theta(t)^2} > 0$$
	   and note that
	   $$ \min_{t \in [0,T]} \frac{a(t,0)}{\theta(t)^2} =  \min_{t \in [0,A^{-1}(T)]} \frac{a(A(t),0)}{\theta(A(t))^2}. $$
Define
	   \begin{align} \label{eq:cp} a_{cp}(s,x)=\mu_0- K \left( \max_{\tau \in [0,A^{-1}(T)]} \frac{1}{\theta(A(\tau))^2} \right) x, \end{align}	   
	   where $K$ is the Lipschitz constant of $a$ with respect to $x$ from assumption (L). Since
	   $$a(A(s),x) = a(A(s),0) + a(A(s),x)-a(A(s),0)$$
	   we now have that
	   $$ \frac{a(A(s),x)}{\theta^2(A(s))} \geq a_{cp}(s,x), \qquad s \in [0,A^{-1}(T)], \,\ x \in \mathbb{R}.$$
	   Therefore, the comparison result for SDEs, see Proposition \ref{prop:comparison} in the Appendix, yields that
	   $$ P(\widetilde{X}_t \geq Z_t, \, t \in [0,A^{-1}(T)])=1, $$
	   where $Z$ is given by the SDE
	   $$  Z_t = x_0 +\int_0^{t} a_{cp}(s,Z_s) ds + \int_0^{t} \sqrt{Z_{s}^+} d\widetilde{W}_s, \qquad t \in [0,A^{-1}(T)].$$
	  We have
	  $$ P(  Z_t \geq 0, \, t \in [0,A^{-1}(T)])=1 $$
	  and 
	  $$  \sup_{t \in [0,A^{-1}(T)]} \mathbb{E} \left[ Z_t^{p} \right] < \infty \qquad \textrm{for}  \qquad p > - 2 \mu_0, $$
	  see Lemma \ref{lem:classcial}(i).
	  Since
	  $$  \frac{1}{T} \int_0^T  \mathbb{E}   \big[ X_t^{2\sk-1} \big] dt  \leq \sup_{t \in [0,A^{-1}(T)]}  \mathbb{E} \big[ \widetilde{X}_t^{2\sk-1} \big]  \leq  \sup_{t \in [0,A^{-1}(T)]}   \mathbb{E}  \big[ Z_t^{2\sk-1} \big], $$
	  condition \eqref{ass-k} with $\ga=1/2$, that is
	  $$  \mathbb{E} \int_0^T \left( \theta(t)^2 X_t \right)^{2\sk-1} dt < \infty,  $$ 
holds if
$$ \sk+\mu_0 > 1/2.$$
Hence, Theorem  \ref{thm:up-bound}  yields convergence order $\lambda$ for all $\lambda <1/2$, if $\m_0 >1/2$. Otherwise, i.e., if $\mu_0 \leq 1/2$, Theorem \ref{thm:up-bound}  yields convergence order $\lambda$ for all $\lambda < \mu_0$.

	   (3) For assertion (ii) note that
	  $$ \sup_{t \in [0,T]} \mathbb{E} \left[  \left(\theta(t)\sqrt{ X_t(1-X_t)^+} \right )^{p} \right]< \infty \quad \Longleftrightarrow  \sup_{t \in [0,A^{-1}(T)]} \mathbb{E} \left[\left(\widetilde{X}_t(1-\widetilde{X}_t) \right)^{p/2} \right] < \infty $$
	  for $p \in \mathbb{R}$ where 
	  \begin{align*} \widetilde{X}_t&= x_0 +\int_0^{t} \frac{a(A(s),\widetilde{X}_{s})}{\theta(A(s))^2}ds + \int_0^{t} \sqrt{(\widetilde{X}_{s}(1-\widetilde{X}_s))^+} d\widetilde{W}_s, \qquad t \in [0,A^{-1}(T)].  \end{align*}
(a) Proceeding analogously as for case (i) we obtain that  
	  $$ P( \widetilde{X}_t \geq Z_t \geq 0, \, t \in [0,A^{-1}(T)])=1 $$
	where
 $Z$ is given by the SDE
	$$  Z_t = x_0 +\int_0^{t} a_{cp}(s,Z_s) ds + \int_0^{t} \sqrt{(Z_{s}(1-Z_s))^+} d\widetilde{W}_s, \qquad t \in [0,A^{-1}(T)],$$
	 with $a_{cp}$ as in \eqref{eq:cp} and 
	  $$  \sup_{t \in [0,A^{-1}(T)]} \mathbb{E} \left[ Z_t^{p} \right] < \infty \qquad \textrm{for }  \quad p > - 2 \mu_0,  $$
	   see Lemma \ref{lem:classcial}(ii), which yields that
	    $$  \mathbb{E} \int_0^T \left( \theta(t)^2 X_t \right)^{2\sk-1} dt < \infty $$ 
holds if
	   $$ \sk+\mu_0 > 1/2.$$ Here we have used again 
	   $$  \frac{1}{T} \int_0^T  \mathbb{E}   \big[ X_t^{2\sk-1} \big] dt  \leq \sup_{t \in [0,A^{-1}(T)]}  \mathbb{E} \big[ \widetilde{X}_t^{2\sk-1} \big]  \leq  \sup_{t \in [0,A^{-1}(T)]}   \mathbb{E}  \big[ Z_t^{2\sk-1} \big].$$
(b) Now we need to establish that
 $$  \mathbb{E} \int_0^T \left( \theta(t)^2 (1-X_t) \right)^{2\sk-1} dt < \infty  $$ 
holds if 
$$ \sk+\mu_1 > 1/2$$
   where
	   $$ \mu_1 := - \max_{t \in [0,T]} \frac{a(t,1)}{\theta(t)^2}.$$
This can be done, e.g., by using the reflected process  $\widetilde{X}_t^r=1-\widetilde{X}_t$, which satisfies the SDE
	\begin{align*} \widetilde{X}_t^r&= 1-x_0 -\int_0^{t} \frac{a(A(s),1-\widetilde{X}^r_{s})}{\theta(A(s))^2}ds + \int_0^{t} \sqrt{(\widetilde{X}_{s}^r(1-\widetilde{X}^r_s))^+} d\widetilde{W}^r_s, \qquad t \in [0,A^{-1}(T)], \end{align*}
	where $\widetilde{W}^r$ is a Brownian motion given by
	$$ \widetilde{W}^r_t = - \widetilde{W}_t, \qquad t \in [0,A^{-1}(T)],$$
	  and
	   $$ a_{cp}(s,x)=\mu_1- K \left( \max_{\tau \in [0,A^{-1}(T)]} \frac{1}{\theta(A(\tau))^2} \right) x. $$	   
	   Here $K$ is again the Lipschitz constant of $a$ with respect to $x$ from assumption (L). Since
	   $$a(A(s),1-x) = a(A(s),1) + a(A(s),1-x)-a(A(s),1)$$
	   we have that
	   $$ - \frac{a(A(s),1-x)}{\theta(A(s))^2} \geq a_{cp}(s,x), \qquad s \in [0,A^{-1}(T)], \,\ x \in \mathbb{R}.$$
	   Therefore, another application of the comparison result for SDEs, see Proposition \ref{prop:comparison} in the Appendix, yields that
	   $$ P(\widetilde{X}_t^r \geq Z_t \geq 0, \, t \in [0,A^{-1}(T)])=1, $$
	   where $Z$ is given by the SDE
	   $$  Z_t = x_0 +\int_0^{t} a_{cp}(s,Z_s) ds + \int_0^{t}  \sqrt{(Z_{s}(1-Z_s))^+} d\widetilde{W}^r_s , \qquad t \in [0,A^{-1}(T)].$$
	   Now we have
	   $$  \sup_{t \in [0,A^{-1}(T)]} \mathbb{E} \left[ Z_t^{p} \right] < \infty \qquad \textrm{for }  \quad p > - 2 \mu_1,  $$
	   see Lemma \ref{lem:classcial}(ii), which yields that
	   \begin{align*}  \mathbb{E} \int_0^T \left( \theta(t)^2 (1-X_t) \right)^{2\sk-1} dt < \infty,  \end{align*} 
	  if
	   $$ \sk+\mu_1 > 1/2.$$

	  (c)  Since
	  \begin{align*}
	   (x(1-x))^+)^{2 \sk -1}
	  	&\le  2^{1-2\sk} \left( (x^+)^{2 \sk -1} \mathbbm{1}_{ \{  x \leq 1/2 \} } + ((1-x)^+)^{2 \sk -1} \mathbbm{1}_{ \{ x >1/2 \} } \right) \\ & \le  2^{1-2\sk} \left( (x^+)^{2 \sk -1}   + ((1-x)^+)^{2 \sk -1}  \right),
	  \end{align*} we obtain from (a), (b) and the Minkowski inequality that
  $$  \mathbb{E} \int_0^T \left( \theta(t)^2 X_t(1-X_t) \right)^{2\sk-1} dt< \infty$$
  if  $\sk+\mu_i >1/2$ for $i=0,1$. Assertion (ii) now follows from Theorem \ref{thm:up-bound}.
	\end{proof}

Thus, the relative strength of the upward and downward drift, respectively,  with respect to $\theta$ at the boundaries of the domain of the SDE influences the convergence rate for these prototype equations with a square root diffusion coefficient.

\smallskip

However, for $c_3$ already $\mu_0>0$ is  sufficient  to obtain convergence order $1/2-\varepsilon$.
\begin{proposition}\label{prop:cklas}
 Let $x_0 >0$, $\ga \in (1/2,1)$ and $\theta \in \mathcal{C}^{1/2}([0,T];(0,\infty))$. Moreover, assume that the drift coefficient
	$a$ satisfies	(HLG), (L)  and assume furthermore that $$ \mu_0 := \min_{t \in [0,T]} \frac{a(t,0)}{\theta(t)^2} > 0.$$
	Then the unique strong solution of 
	$$   	dX_t = a(t,X_t)dt + \theta(t) \left( X_t^+\right)^{\ga}dW_t,  \quad t\in[0,T], \qquad X_0=x_0,$$
	satisfies $ P( X_t \geq 0, t \in [0,T])=1$ and we have
	\begin{equation*}
		\limsup_{N \rightarrow \infty} \,\, N^{\lambda} \, \sup_{k=0, \ldots, N} \mathbb{E} \left| X_{t_k}- x_k^{(N)} \right| =0	
	\end{equation*}		
	for all $$\lambda <1/2.$$
\end{proposition}
\begin{proof}
	The proof can be done along the same lines as the proof of case (i) of the previous proposition. The only and crucial difference is, that  Lemma \ref{lem:classcial}(iii) gives us the existence off all inverse moments of the CKLS process. Thus we can take always $\sk=0$ in Theorem \ref{thm:up-bound}.
\end{proof}

	\section{Proof of Theorem  \ref{thm:up-bound} }\label{proof-main-result}
	\subsection{Preliminaries}
	Before we start with the proof of our main result, we will present some results that will help us later. First, we define a time-continuous version of the Euler scheme  by $\bar{x}^{(N)}_t=\bar{x}_t$ with
	\begin{equation}\label{euler-cont}
		\bar{x}_t = \bar{x}_{\eta(t)} + a(\eta(t),\bar{x}_{\eta(t)})(t-\eta(t)) + \sigma(\eta(t),\bar{x}_{\eta(t)})^{\ga}(W_t-W_{\eta(t)}), \qquad t \in [0,T].
	\end{equation}
	Here, we denote $\eta(t)=\max\{t_k\in\{t_0,t_1,...,t_N\}: t_k\le t\}$. Clearly, we have $x_k=\overline{x}_{t_k}$ for $k=0, \ldots, N$.

	Note that (HLG) and (L) imply that $a$ and $\sigma$ satisfy a linear growth condition, that is there exists $C>0$ such that
	$$ |a(t,x)|+|\sigma(t,x)| \leq C (1+|x|) \qquad \textrm{for all} \quad  t\in [0,T],\,  x\in \mathbb{R}.$$
	Since 
	$$ x^{\ga} \leq 1 + x \qquad \textrm{for all} \quad x\in [0, \infty),$$
	also $\sigma^{\ga}$ satisfies a linear growth condition. Thus, we obtain the following result by standard computations.

	\begin{lemma}\label{bounded}
		Let $p\ge 1$ and Assumption \ref{ass-1} hold. Then, there exist constants $C_p>0$ such that
		\begin{align*}
			\mathbb{E}\left[  \sup_{t\in [0,T]} |X_t|^p \right]  + \sup_{0 \leq s< t \leq  T } \mathbb{E} \left[  \frac{|X_t-X_s|^p}{|t-s|^{p/2}} \right] \leq C_p
		\end{align*}
	and 
	\begin{displaymath}
	\sup_{N \in \mathbb{N}}	\mathbb{E}\left[  \sup_{t\in [0,T]} |\bar{x}_t^{(N)}|^p \right]  + \sup_{N \in \mathbb{N}}	 \sup_{0 \leq s< t \leq  T } \mathbb{E} \left[ \frac{ | \bar{x}_t^{(N)} - \bar{x}_s^{(N)}|^p}{|t-s|^{p/2}}  \right]  \leq C_p.
	\end{displaymath}

\end{lemma}
	
	The next lemma gives us a bound for the  expected local time in zero of a semimartingale. It is taken from \cite{DA}.
	\begin{lemma}\label{lemmadeAngelis}
		For any $\delta\in(0,1)$ and any real-valued, continuous semimartingale $Y=(Y_t)_{t \in [0,T]}$, we have
		\begin{displaymath}
			\begin{aligned}
				\mathbb{E}\left[L^0_t(Y)\right]\le 4\delta&-2\mathbb{E}\left[\int_{0}^{t}\left(\mathbbm{1}_{\{Y_s\in(0,\delta)\}}+\mathbbm{1}_{\{Y_s \geq \delta\}}e^{1-\frac{Y_s}{\delta}}\right)dY_s\right]\\
				&+\frac{1}{\delta}\mathbb{E}\left[\int_{0}^{t}\mathbbm{1}_{\{Y_s>\delta\}}e^{1-\frac{Y_s}{\delta}}d\langle Y\rangle_s\right], \qquad  \qquad t \in [0,T].
			\end{aligned}
		\end{displaymath}
	\end{lemma}

Here $(L_t^0(Y))_{t\in [0,T]}$ is the local time  of $Y$ in $x=0$. For almost all $\omega \in \Omega$  the map $[0,T] \ni t\mapsto [L_t^0(Y)](\omega) \in \mathbb{R}$ is continuous and non-decreasing with $L_0^0(Y)=0$. 
For more background on local times see, e.g., Chapter III.7 in \cite{KS}. For our purposes, it suffices to mention the Tanaka-Meyer formula for continuous semimartingales of the form
$$Y_t=y_0+ \int_0^t y^{A}_s ds + \int_0^t y_s^M dW_s, \qquad t \in[0,T],$$
where $(y_t^A)_{t \in [0,T]}$ and $(y_t^M)_{t \in [0,T]}$ are  continuous, square-integrable and adapted processes and $y_0 \in \mathbb{R}$. Then the Tanaka-Meyer formula, see, e.g., Equation 7.9 in Chapter III in \cite{KS}, states that
\begin{equation}
	\label{TM-formula}
	|Y_t|= |y_0| + \int_0^t \sign(Y_s) y_s^A ds +  \int_0^t \sign(Y_s) y_s^M dW_s +2 L_t^0(Y), \qquad t \in [0,T]. 
\end{equation}

Finally,	the following inequalities will be very useful in our computations.
	\begin{lemma}\label{squareroot}
		Let $\ga \in \left[\frac{1}{2},1\right)$ and $ \beta \in (0,1]$.
 Then, we have
		\begin{equation}\label{ineq_1}
			\begin{aligned}
				\left|x^\ga-y^{\ga}\right| \leq 2
				x^{-(1-\ga)}\left|x-y\right| \qquad \textrm{for all } \quad  x>0, \,  y\ge0,
			\end{aligned}
		\end{equation}
	and 
		\begin{equation}\label{ineq_2}
			\begin{aligned}
				\left|x^\ga-y^{\ga}\right|\le 2
				x^{-(1-\ga)\beta}\left|x-y\right|^{\beta +\ga(1-\beta)}  \qquad \textrm{for all } \quad  x>0, \, y\ge0.
			\end{aligned}
		\end{equation}
	\end{lemma}
	\begin{proof}
	(i)	We first prove the assertion
		\begin{equation*}
			|x^\ga-y^\ga|\left(x^{1-\ga}+y^{1-\ga}\right)\le 2|x-y|
		\end{equation*}
	from which inequality \eqref{ineq_1} directly follows. For $x=y$, the statement holds trivially. Without loss of generality, we can assume that $x>y$ and prove the assertion
	\begin{equation*}
		(x^\ga-y^\ga)\left(x^{1-\ga}+y^{1-\ga}\right)\le 2(x-y).
	\end{equation*}
	We start with
	\begin{align*}
		(x^\ga-y^\ga)\left(x^{1-\ga}+y^{1-\ga}\right)&=y^\ga\left(\left(\frac{x}{y}\right)^\ga-1\right)y^{1-\ga}\left(\left(\frac{x}{y}\right)^{1-\ga}+1\right)\\
		&=y\left(\left(\frac{x}{y}\right)^\ga-1\right)\left(\left(\frac{x}{y}\right)^{1-\ga}+1\right).
	\end{align*}
	Now set $z:=\frac{x}{y}$. Notice that $z>1$. Therefore, it remains to show that
	\begin{align}
		(z^\ga-1)\left(z^{1-\ga}+1\right)\le 2(z-1)  \label{step1-sqr}
	\end{align} for $z>1$.
	Next, we define
	\begin{align*}
		f(z):=2(z-1)-(z^\ga-1)\left(z^{1-\ga}+1\right).
	\end{align*}
We have $f(1)=0$ and if we are able to show that $f$ is increasing for all $z \geq 1$, then Equation \eqref{step1-sqr}
follows. We have	\begin{align*}
		f'(z)=1-\ga z^{\ga-1}+(1-\ga)z^{-\ga}
	\end{align*}
and consequently
$$ 	f'(z)\geq 1-\ga >0$$ for all $ z\geq 1$. This concludes the proof of \eqref{step1-sqr} and the first assertion.

(ii) For the second assertion note that
$$ |x^{\ga}-y^{\ga}| \leq  |x-y|^{\ga} $$
by concavity. Equation \eqref{ineq_2} follows then from
$$ |x^{\ga}-y^{\ga}|=|x^{\ga}-y^{\ga}|^{\beta}|x^{\ga}-y^{\ga}|^{1-\beta} \leq |x^{\ga}-y^{\ga}|^{\beta} |x-y|^{\ga(1-\beta)} $$ and inequality \eqref{ineq_1}.
	\end{proof}
	\subsection{Proof of Theorem \ref{thm:up-bound}}
	
	 Theorem  \ref{thm:up-bound} is a direct consequence of the following Theorem \ref{thm:one} with
	 $$\beta=1-\frac{1}{1-\ga} \left( \sk + \frac{\varepsilon}{2} \right)$$ and $\varepsilon >0$ arbitrarily small.
	 The strategy of the proof of Theorem \ref{thm:one} is based on the strategy of \cite{MiNe}, which  in turn was inspired by \cite{DA}.
	
	\begin{theorem}\label{thm:one}
		Let $\bar{x}$ be given as in \eqref{euler-cont} and let Assumption \ref{ass-1} be fulfilled.
		Moreover, assume that there exists $\beta \in (0,1]$ and $\varepsilon >0$ such that
		\begin{equation} \label{ass-k-proof}
			\mathbb{E}\left[\int_{0}^{T} 	\sigma(s,X_s)^{2(\ga-1)\beta-\varepsilon }	ds\right] < \infty.
		\end{equation}
		Then, we have  
		\begin{displaymath}
		\limsup_{N \rightarrow \infty}  \, N^{\lambda}	\sup\limits_{t\in[0,T]}\mathbb{E} \left|X_t-\bar{x}_t^{(N)}\right|  =0
		\end{displaymath}
	for all $\lambda < \ga -1/2 + \beta(1-\ga)$.	\end{theorem}
	\begin{proof}
		Define the error process $R=(R_t)_{t \in [0,T]}$ by $R_t=X_t-\bar{x}_t$. 
		
		(i) The Tanaka-Meyer formula, see Equation \eqref{TM-formula}, yields
		\begin{equation}\label{prooffirststep}
			\begin{aligned}
				\mathbb{E}\left[ \left|R_t\right|\right]  =& \mathbb{E}\left[ \int_{0}^{t}\sign(R_u)dR_u\right]+\mathbb{E}\left[L^0_t(R)\right]\\
				=&\mathbb{E}\left[\int_{0}^{t}\sign(R_u)\left(a(u,X_u)-a(\eta(u),\bar{x}_{\eta(u)})\right)du\right] \\ & +\mathbb{E}\left[\int_{0}^{t}\sign(R_u)\left(\sigma(u,X_u)^{\ga}-\sigma(\eta(u),\bar{x}_{\eta(u)})^\ga\right)dW_u\right]
				+ \mathbb{E} \left[L^0_t(R) \right].
			\end{aligned} 
		\end{equation}
		We have
		\begin{displaymath}\mathbb{E}\left[\int_{0}^{t}\sign(R_u)\left(\sigma(u,X_u)^{\ga}-\sigma(\eta(u),\bar{x}_{\eta(u)})^\ga\right)dW_u\right] =0
		\end{displaymath}
		due to Lemma \ref{bounded}, the linear growth of the diffusion coefficient and the martingale property of the It\=o integral.
		Looking at the first term, we have
		\begin{displaymath}
			\begin{aligned}
			& 	\mathbb{E}\left[\int_{0}^{t}\sign(R_u)\left(a(u,X_u)-a(\eta(u),\bar{x}_{\eta(u)})\right)du\right] \\ \quad &= \mathbb{E}\left[\int_{0}^{t}\sign(R_u)\left(a(u,X_u)-a(u,\bar{x}_{u})\right)du\right] \\
				&\qquad +  \mathbb{E}\left[\int_{0}^{t}\sign(R_u)\left(a(u,\bar{x}_{u})-a(\eta(u),\bar{x}_{u})\right)du\right] \\
					&\qquad +  \mathbb{E}\left[\int_{0}^{t}\sign(R_u)\left( a(\eta(u),\bar{x}_u)-a(\eta(u),\bar{x}_{\eta(u)})\right)du\right] 
				\\  \quad & \leq  
				 C \left( \int_{0}^{t} \mathbb{E}\left[ |R_u|\right]  du	+
				( \Delta t )^{1/2} \int_{0}^{t} \mathbb{E}\left[ 1+ |\bar{x}_u|\right] du + \int_{0}^{t} \mathbb{E}\left[ |\bar{x}_u-\bar{x}_{\eta(u)}|\right]  du \right)
			\end{aligned}
		\end{displaymath}
	using (HLG) and (L). Lemma \ref{bounded}  now implies
		\begin{equation}\label{cir-tanaka-1}
			\sup_{u \in [0,t]} \mathbb{E}\left[ \left|R_u\right|\right]  \leq  C(\Delta t)^{\frac{1}{2}} +  C \int_{0}^{t}  	\sup_{v \in [0,u]} \mathbb{E}\left[ |R_v|\right]  du + \mathbb{E}\left[L^0_t(R)\right].
		\end{equation}
		
		(ii)	With Lemma \ref{lemmadeAngelis} we can derive a bound for the expected local time of the error process $R$ in zero. Let $\delta\in(0,1)$, then
		\begin{equation}\label{local}
			\begin{aligned}
				\mathbb{E}\left[L^0_t(R)\right]\le 4\delta&-2\mathbb{E}\left[\int_{0}^{t}\left(\mathbbm{1}_{\{R_s\in(0,\delta)\}}+\mathbbm{1}_{\{R_s \geq \delta\}}e^{1-\frac{R_s}{\delta}}\right)dR_s\right]\\
				&+\frac{1}{\delta}\mathbb{E}\left[\int_{0}^{t}\mathbbm{1}_{\{R_s>\delta\}}e^{1-\frac{R_s}{\delta}}d\langle R \rangle_s\right].
			\end{aligned}
		\end{equation}
		We define $Y_s:=\mathbbm{1}_{\{R_s\in(0,\delta)\}}+\mathbbm{1}_{\{R_s \geq \delta\}}e^{1-\frac{R_s}{\delta}}$ and look at the second term of \eqref{local}, i.e.,
		\begin{displaymath}
			\mathbb{E}\left[\int_{0}^{t}Y_sdR_s\right]= \,  \mathbb{E}\left[\int_{0}^{t}Y_s\left(a(s,X_s)-a(\eta(s),\bar{x}_{\eta(s)})\right) ds\right],
		\end{displaymath}
		where we already used the martingale property of the It\=o integral. Since $0 \leq Y_s \leq 1$, we obtain proceeding as above  that 
		\begin{equation}
			\begin{aligned} \label{LT-3}
				\left| 	\mathbb{E}\left[\int_{0}^{t}Y_sdR_s\right] \right|  \leq 
				C(\Delta t)^{\frac{1}{2}} +  C \int_{0}^{t} \mathbb{E}\left[ |R_u| \right] du.
			\end{aligned}
		\end{equation}
	
	For the third term of  Equation \eqref{local}  note that
	$$ \langle R \rangle_t= \int_0^t \left(\sigma(s,X_s)^{\ga}-\sigma(\eta(s),\bar{x}_{\eta(s)})^{\ga}\right)^2 ds
, \qquad t \in [0,T].$$

Condition \eqref{ass-k-proof} implies that
$$ P(\sigma(x,X_s)=0)=0 $$
for all $s \in [0,T]$.
Thus, we can apply Lemma \ref{squareroot}, i.e.,
$$
	\left|x^\ga-y^{\ga}\right|\le 2
	x^{-(1-\ga)\beta}\left|x-y\right|^{\beta +\ga(1-\beta)}, \qquad x>0, \, y \geq 0,
$$ 
and obtain
		\begin{displaymath}
			\begin{aligned}
				\frac{1}{\delta}\mathbb{E}\left[\int_{0}^{t}\mathbbm{1}_{\{R_s>\delta\}}e^{1-\frac{R_s}{\delta}}d\langle R \rangle_s\right]&=\frac{1}{\delta}\mathbb{E}\left[\int_{0}^{t}\mathbbm{1}_{\{R_s>\delta\}}e^{1-\frac{R_s}{\delta}} \left(\sigma(s,X_s)^{\ga}-\sigma(\eta(s),\bar{x}_{\eta(s)})^{\ga}\right)^2 ds\right]\\
				&\le \frac{4}{\delta}\mathbb{E}\left[\int_{0}^{t}\mathbbm{1}_{\{R_s>\delta\}}e^{1-\frac{R_s}{\delta}}
			\frac{ |\sigma(s,X_s)-\sigma(\eta(s),\bar{x}_{\eta(s)})|^{2\beta +2 \ga(1-\beta)}}{ 	\sigma(s,X_s)^{2(1-\ga)\beta} }
				ds\right].
		\end{aligned}
	\end{displaymath}
Since 
$$ |x|+|y|+|z| \leq C_{\alpha}\left(  |x|^{\alpha} +|y|^{\alpha} +|z|^{\alpha}\right)^{1/\alpha}\qquad \textrm{for all} \quad x,y,z \in \mathbb{R},$$
for $\alpha \geq 1 $ and
$$ 2\beta +2 \ga(1-\beta) = 2(1-\ga)\beta + 2\ga,$$ it follows that
 	\begin{displaymath}
 	\begin{aligned}
 		\frac{1}{\delta}\mathbb{E}\left[\int_{0}^{t}\mathbbm{1}_{\{R_s>\delta\}}e^{1-\frac{R_s}{\delta}}d\langle R \rangle_s\right]
 		&\le \frac{C_{\beta,\ga}}{\delta}\mathbb{E}\left[\int_{0}^{t}\mathbbm{1}_{\{R_s>\delta\}}e^{1-\frac{R_s}{\delta}}
 		\frac{ |\sigma(s,X_s)-\sigma(s,\bar{x}_{s})|^{2(1-\ga)\beta + 2\ga}}{ 	\sigma(s,X_s)^{2(1-\ga)\beta} }
 		ds\right] \\  & \qquad +  \frac{C_{\beta,\ga}}{\delta}\mathbb{E}\left[\int_{0}^{t}\mathbbm{1}_{\{R_s>\delta\}}e^{1-\frac{R_s}{\delta}}
 		\frac{ |\sigma(s,\bar{x}_{s})-\sigma(\eta(s),\bar{x}_{s}) |^{2(1-\ga)\beta + 2\ga}}{ 	\sigma(s,X_s)^{2(1-\ga)\beta} }	ds\right]
 		\\  & \qquad +  \frac{C_{\beta,\ga}}{\delta}\mathbb{E}\left[\int_{0}^{t}\mathbbm{1}_{\{R_s>\delta\}}e^{1-\frac{R_s}{\delta}}
 		\frac{ |\sigma(\eta(s),\bar{x}_{s}) - \sigma(\eta(s),\bar{x}_{\eta(s)}) |^{2(1-\ga)\beta + 2\ga}}{ 	\sigma(s,X_s)^{2(1-\ga)\beta} }	ds\right].
 	\end{aligned}
 \end{displaymath}
Moreover, exploiting (HLG) and (L) yields 
	\begin{displaymath}
	\begin{aligned}
		\frac{1}{\delta}\mathbb{E}\left[\int_{0}^{t}\mathbbm{1}_{\{R_s>\delta\}}e^{1-\frac{R_s}{\delta}}d\langle R\rangle_s\right]		&\le \frac{C_{\beta,\ga}}{\delta}\mathbb{E}\left[\int_{0}^{t}\mathbbm{1}_{\{R_s>\delta\}}e^{1-\frac{R_s}{\delta}}
			\frac{ |R_s|^{2(1-\ga)\beta + 2\ga}}{ 	\sigma(s,X_s)^{2(1-\ga)\beta} }	 ds\right] \\
			& \qquad + \frac{C_{\beta,\ga}}{\delta} (\Delta t)^{\beta(1-\ga) +\ga} \, \mathbb{E}\left[\int_{0}^{t}\mathbbm{1}_{\{R_s>\delta\}}e^{1-\frac{R_s}{\delta}}
		\frac{ (1+|\bar{x}_s|)^{2(1-\ga)\beta + 2\ga}}{ 	\sigma(s,X_s)^{2(1-\ga)\beta} }	 ds\right]	
		\\
		& \qquad + \frac{C_{\beta,\ga}}{\delta}  \mathbb{E}\left[\int_{0}^{t}\mathbbm{1}_{\{R_s>\delta\}}e^{1-\frac{R_s}{\delta}}
		\frac{ |\bar{x}_s-\bar{x}_{\eta(s)}|^{2(1-\ga)\beta + 2\ga}}{ 	\sigma(s,X_s)^{2(1-\ga)\beta} }	 ds\right]	.
			\end{aligned}
		\end{displaymath}
Applying	Lemma \ref{bounded}, condition \eqref{ass-k-proof} and H\"older's inequality with
	$ p= \frac{2(1-\ga)\beta +\varepsilon}{2(1-\ga)\beta}$ and $q=\frac{p}{p-1}$ to the second and third term on the right side now gives
	\begin{equation} \label{step_main}
	\begin{aligned}
			\frac{1}{\delta}\mathbb{E}\left[\int_{0}^{t}\mathbbm{1}_{\{R_s>\delta\}}e^{1-\frac{R_s}{\delta}}d\langle R \rangle_s\right]		&\le \frac{C_{\beta,\ga}}{\delta}\mathbb{E}\left[\int_{0}^{t}\mathbbm{1}_{\{R_s>\delta\}}e^{1-\frac{R_s}{\delta}}
			\frac{ |R_s|^{2(1-\ga)\beta + 2\ga}}{ 	\sigma(s,X_s)^{2(1-\ga)\beta} }	 ds\right] 
		\\ & \qquad  + \frac{C_{\beta,\ga,\varepsilon}}{\delta} (\Delta t)^{\beta(1-\ga) +\ga} . 
		\end{aligned} \end{equation}

	(iii) It remains to study  \begin{displaymath}
		\begin{aligned}
		 \frac{C_{\beta,\ga}}{\delta}\mathbb{E}\left[\int_{0}^{t}\mathbbm{1}_{\{R_s>\delta\}}e^{1-\frac{R_s}{\delta}}
			\frac{ |R_s|^{2(1-\ga)\beta + 2\ga}}{ 	\sigma(s,X_s)^{2(1-\ga)\beta} }	 ds\right].	
		\end{aligned}
	\end{displaymath}
		For this let $\alpha\in(0,1)$. Since
		\begin{displaymath}
			\sup_{s \in [0,T]} \mathbb{E}\left[ |R_s|^p\right]  \leq C_p
		\end{displaymath}
		for all $p \geq 1$ due to our assumptions and Lemma \ref{bounded}, we have that
		\begin{equation*}
					\begin{aligned}
				&	\frac{1}{\delta}\mathbb{E}\left[\int_{0}^{t}\mathbbm{1}_{\{R_s>\delta\}}e^{1-\frac{R_s}{\delta}}
					\frac{ |R_s|^{2(1-\ga)\beta + 2\ga}}{ 	\sigma(s,X_s)^{2(1-\ga)\beta} }	 ds\right]	
		 \\ & \quad =\frac{1}{\delta}\mathbb{E}\left[\int_{0}^{t}\mathbbm{1}_{\{R_s\in(\delta,\delta^{\alpha})\}}e^{1-\frac{R_s}{\delta}}	\frac{ |R_s|^{2(1-\ga)\beta + 2\ga}}{ 	\sigma(s,X_s)^{2(1-\ga)\beta} }	 ds\right] \\ & \qquad \quad  +\frac{1}{\delta}\mathbb{E}\left[\int_{0}^{t}\mathbbm{1}_{\{R_s\ge\delta^{\alpha}\}}e^{1-\frac{R_s}{\delta}} 	\frac{ |R_s|^{2(1-\ga)\beta + 2\ga}}{ 	\sigma(s,X_s)^{2(1-\ga)\beta} }	 ds\right]\\
				& \quad \le \delta^{ \alpha(2(1-\ga)\beta + 2\ga)-1}\int_{0}^{t}\mathbb{E} \left[ \sigma(s,X_s)^{-2(1-\ga)\beta}\right]ds\\ & \qquad \quad  +C_{\beta, \ga, \varepsilon} \frac{e^{-\delta^{\alpha-1}}}{\delta}\int_{0}^{t} 	 \left( \mathbb{E} \left[ \sigma(s,X_s)^{-2(1-\ga)\beta-\varepsilon } \right] \right)^{\frac{2(1-\ga)\beta}{2(1-\ga)\beta+\varepsilon}}ds
			\end{aligned}
		\end{equation*}
		by another application of the H\"older inequality. Since 
		$$ \limsup_{\delta \rightarrow 0} \frac{e^{-\delta^{\alpha-1}}}{\delta^{p}} =0$$ for all $p\ge 0$, condition \eqref{ass-k-proof} now yields
		\begin{equation}
		\begin{aligned}
			&	\frac{1}{\delta}\mathbb{E}\left[\int_{0}^{t}\mathbbm{1}_{\{R_s>\delta\}}e^{1-\frac{R_s}{\delta}}
			\frac{ |R_s|^{2(1-\ga)\beta + 2\ga}}{ 	\sigma(s,X_s)^{2(1-\ga)\beta} }	 ds\right]	
		\leq C_{\beta,\ga, \alpha, \varepsilon }\delta^{ \alpha(2(1-\ga)\beta + 2\ga)-1}.
		 \label{LT-1}
			\end{aligned}
	\end{equation}

		Summarizing  the Equations \eqref{local} -- \eqref{LT-1} we have shown that
		\begin{equation}  \begin{aligned}\label{LT-0}
			\mathbb{E}\left[L^0_t(R)\right]  \leq 4 \delta & +	C(\Delta t)^{\frac{1}{2}} +  C \int_{0}^{t} \mathbb{E}\left[ |R_u|\right]  du \\ &  +  C_{\beta,\ga, \alpha , \varepsilon}\delta^{ \alpha(2(1-\ga)\beta + 2\ga)-1}+ \frac{C_{\beta,\ga,\varepsilon}}{\delta} (\Delta t)^{\beta(1-\ga) +\ga}. \end{aligned}
		\end{equation}
		(iv) Setting $\delta = (\Delta t)^{1/2}$ gives
		\begin{equation*} 
				\mathbb{E}\left[L^0_t(R)\right]  \leq    C  (\Delta t)^{1/2} + C \int_{0}^{t} \mathbb{E}\left[ |R_u| \right] du +  C_{\beta,\ga, \alpha,  \varepsilon }(\Delta t)^{ \alpha((1-\ga)\beta + \ga)-\frac{1}{2}}
				+ C_{\beta,\ga,\varepsilon} (\Delta t)^{\beta(1-\ga) +\ga-\frac{1}{2}}. 
		\end{equation*}	 
	Note that
	$$  \alpha (\beta(1-\ga)+\ga) - \frac{1}{2} < \beta(1-\ga)+\ga- \frac{1}{2} \leq \frac{1}{2}, $$
	since we have $\alpha  \in (0,1)$ and $\beta \in (0,1]$. Consequently, we have
		\begin{equation*} 
		\mathbb{E}\left[L^0_t(R)\right]  \leq     C \int_{0}^{t} \mathbb{E}\left[ |R_u| \right] du 
		+ C_{\beta,\ga, \alpha, \varepsilon} (\Delta t)^{ \alpha (\beta(1-\ga) +\ga)-\frac{1}{2}}. 
	\end{equation*}	
		Combining this with Equation \eqref{cir-tanaka-1} yields
		\begin{displaymath}
			\sup_{u \in [0,t]}	\mathbb{E}\left[ \left|R_u\right|\right]  \leq  C_{\beta,\ga, \alpha, \varepsilon} (\Delta t)^{ \alpha (\beta(1-\ga) +\ga)-\frac{1}{2}}  +  C \int_{0}^{t} \sup_{v \in [0,u]}\mathbb{E}\left[   |R_v|\right]  du 
		\end{displaymath}
		and the assertion follows  by choosing $\alpha \in (0,1)$ sufficiently large and an application of Gronwall's lemma.
		
	\end{proof}

	\section{Appendix}
	
For the sake of convenience, we collect here some general auxiliary results.

	 \subsection{Feller test}
	 Consider the autonomous SDE
	 		\begin{equation}\label{SDE-auto-F}
	 		dX_t = a(X_t)dt + c(X_t)dW_t,  \quad t\in[0,T], \qquad X_0=x_0.
	 	\end{equation}
	 We define
	 \begin{align*}
	 	S := \inf\{t\ge0: X_t\notin(l,r)\}
	 \end{align*}
	 as the exit time from $I=(l,r)$ with $-\infty \leq l <r \leq \infty$. Now, {\it Feller's test for explosions} (see Theorem 5.29 in Chapter V in \cite{KS}) gives necessary and sufficient conditions for the finiteness of $S$.
	 \begin{theorem}\label{thm:feller-test}
	 	Assume that the coefficients $a:I\rightarrow\mathbb{R}$ and $c:I\rightarrow\mathbb{R}$ satisfy
	 	\begin{align*}
	 		\forall x\in I:\quad	c^2(x)>0,
	 	\end{align*}
	 	and
	 	\begin{align*}
	 		\forall x\in I \,\, \,\, \exists \varepsilon>0: \int_{x-\varepsilon}^{x+\varepsilon}\frac{1+|a(y)|}{c^2(y)}dy<\infty.
	 	\end{align*}
	 	Define the {\it scale function} $p$ by
	 	\begin{align*}
	 		p(x) := \int_{o}^{x} \exp\left(-2\int_{o}^{y}\frac{a(z)}{c^2(z)}dz\right)dy, \quad x\in I,
	 	\end{align*}
	 	for a fixed $o\in I$ and the function $v$ by
	 	\begin{align*}
	 		v(x) := \int_{o}^{x} p'(y)\int_{o}^{y}\frac{2}{p'(z)c^2(z)}dzdy, \quad x\in I.
	 	\end{align*}
	 	Let $(X,W)$ be a weak solution of \eqref{SDE-auto-F} in $I$ with deterministic $X_0=x_0\in I$. Then, $P\left(S=\infty\right)=1$ or ${P}\left(S=\infty\right)<1$ according to whether
	 	\begin{align*}
	 		v(l+)=\lim_{x\downarrow l}v(x)=\lim_{x\uparrow r}v(x)=v(r-)=\infty
	 	\end{align*}
	 	or not.
	 \end{theorem}
	\subsection{A deterministic time change}

	The following auxiliary result deals with deterministic time changes for Riemann and It\=o integrals.
	For this, let $\theta \in \mathcal{C}([0,T]; (0,\infty))$ and define
	$$\Theta:[0,T] \rightarrow \mathbb{R}, \qquad \Theta(t)=\int_0^t \theta^2(s) ds. $$
	Since $\theta^2(t)>0$ for all $t \in [0,T]$,  the function $\Theta$ is  strictly increasing. Therefore its inverse function $A=\Theta^{-1}: D \rightarrow [0,T]$  with $ {D}= \Theta([0,T])$ exists. Then,
	$$ \widetilde{W}_t= \int_0^{A(t)} \theta(s)dW_s, \qquad t \in [0,A^{-1}(T)],$$
	is a Brownian motion adapted to the filtration $(\widetilde{\mathcal{F}}_t)_{t \in [0,A^{-1}(T))]}$ with $\widetilde{\mathcal{F}}_t= \mathcal{F}_{A(t)}$. See, e.g., Proposition 4.6 in Chapter III of \cite{KS}.
	Moreover, we have the following:
	\begin{proposition} \label{prop:time-change}
		Consider an continuous and adapted process $Y=(Y_t)_{t \in [0,T]}$ such that 
		$$ \int_0^T \mathbb{E} |Y_s|^2 ds < \infty$$ and let $\theta \in \mathcal{C}([0,T]; (0,\infty))$. Then, we have
		$$ \int_0^{A(t)} Y_s ds=    \int_0^t Y_{A(\tau)} \frac{1}{\theta^2(A(\tau))}d \tau, \qquad t \in [0,A^{-1}(T)], \qquad P-a.s.,$$
		and 
		$$ \int_0^{A(t)} \theta(s) Y_s d W_s=     \int_0^t Y_{A(\tau)} d \widetilde{W}_{\tau}, \qquad t \in [0,A^{-1}(T)], \qquad P-a.s..$$
	\end{proposition}
	\begin{proof}
		The first statement is the classical substitution rule for Riemann integrals, while the second statement follows from Proposition  4.8 in Chapter III of \cite{KS}.
	\end{proof}

	\subsection{Moment bounds for CIR, WF and CKLS}\label{subsec:classical_bounds} 
	
Here we collect moment bounds for the CIR, WF  and CKLS process. Note that under the assumptions below the Feller test implies that with probability one the  CIR and CKLS process are positive, i.e., take values in $[0,\infty)$, while the WF process takes values in $[0,1]$.
	\begin{lemma}\label{lem:classcial}
		\begin{itemize}
			\item[(i)] Consider the CIR process given by  
			\begin{displaymath}
				dV_t=\kappa\left(\lambda-V_t\right)dt+\theta\sqrt{V_t}dW_t \quad t\in[0,T], \qquad V_0=v_0>0,
			\end{displaymath}
			with $\kappa, \lambda, \theta >0$. Then we have
			\begin{displaymath}
				\sup_{s \in[0,T]}	\mathbb{E} |V_s|^p<\infty \quad \textrm{for all} \quad p > - \frac{2\kappa \lambda}{\theta^2}.
			\end{displaymath} 
			\item[(ii)] Consider the WF process given by
			\begin{displaymath}
				dV_t=\kappa\left(\lambda-V_t\right)dt+ \theta\sqrt{V_t(1-V_t)}dW_t \quad t\in[0,T], \qquad V_0=v_0\in (0,1),
			\end{displaymath}
			with $\kappa, \lambda, \theta >0$. Then we have
			\begin{displaymath}
				\sup_{s \in[0,T]}\mathbb{E}  |V_s|^p <\infty \quad \textrm{for all} \quad p > -   \frac{ 2\kappa \lambda}{\theta^2}  \
			\end{displaymath} and
			\begin{displaymath}
			\sup_{s \in[0,T]}\mathbb{E} |1-V_s|^p  <\infty \quad \textrm{for all} \quad p > -   \frac{ 2\kappa}{\theta^2} (1-\lambda) .
		\end{displaymath} 
			\item[(iii)] Let $\ga \in (1/2,1)$. Consider the CKLS process given by
			\begin{displaymath}
				dV_t=\kappa\left(\lambda-V_t\right)dt+ \theta V_t^{\ga}dW_t \quad t\in[0,T], \qquad V_0=v_0>0,
			\end{displaymath}
			with $\kappa, \lambda, \theta >0$. Then we have
			\begin{displaymath}
				\sup_{s \in[0,T]}\mathbb{E}  |V_s|^p <\infty \quad \textrm{for all} \quad p \in \mathbb{R} .
			\end{displaymath} 
		\end{itemize}	
	\end{lemma}
	\begin{proof} (i) This statement can be found, e.g., in  Section 3 in \cite{DNS} and Theorem 3.1 in \cite{HK}, respectively. \\
		(ii)  This statement can be found, e.g., in Section 3.5 in \cite{NS}.
		(iii) This statement can be found, e.g., in Section 3.4 in \cite{NS} and \cite{BER}, respectively.

	\end{proof}
	
	 \subsection{Comparison result}
	
	Proposition 2.18 in Chapter V of \cite{KS} gives a comparison result for one-dimensional SDEs, which we use in Section \ref{sec:examples}.
	\begin{proposition} \label{prop:comparison}
		Consider two continuous, adapted process $X^{(j)}, j=1,2$, such that
		\begin{align*}
			X^{(j)}_t=	X^{(j)}_0+\int_{0}^{t}a_j(s,X^{(j)}_s)ds + \int_{0}^{t}c(s,X^{(j)}_s)dW_s, \qquad t \in [0,T],
		\end{align*}
	holds $P$-a.s.. Assume that
	\begin{itemize}
		\item[(i)] the  coefficients $c(t,x), a_1(t,x), a_2(t,x)$ are continuous and real-valued functions on $[0,T]\times\mathbb{R}$,
		\item[(ii)] the diffusion coefficient $c(t,x)$ satisfies
		\begin{align*}
			\left|c(t,x)-c(t,y)\right|\le h\left(|x-y|\right) 
		\end{align*}
		for every $t\in[0,T]$, $x,y\in\mathbb{R}$,  where $h:[0,\infty)\rightarrow[0,\infty)$ is strictly increasing with $h(0)=0$ and
		\begin{align*}
			\int_{0}^{\varepsilon}h^{-2}(u)du = \infty\qquad \text{for all} \quad \varepsilon>0,
		\end{align*}
		\item[(iii)] $X^{(1)}_0 \leq X^{(2)}_0$ $P$-a.s.,
		\item[(iv)] $a_1(t,x)\le a_2(t,x)$ for all $t \in [0,T]$ and all $ x\in \mathbb{R}$,
		\item[(v)] either $a_1(t,x)$ or $a_2(t,x)$ satisfies condition (L), i.e., is globally Lipschitz in $x$, uniformly in $t$.
		\end{itemize}
	Then,
		\begin{align*}
		{P}\left(X^{(1)}_t\le X^{(2)}_t \,\, \text{for all } \, \, t\in[0,T] \right) = 1.
		\end{align*}

	\end{proposition}

	\bibliographystyle{plain}

\end{document}